\documentclass[12pt]{amsart}
\usepackage{color}

\usepackage{hyperref}

\usepackage{amssymb}
\usepackage{amsmath}
\usepackage{amsthm}
\usepackage[english]{babel}
\usepackage[latin1]{inputenc}
\usepackage{mathrsfs}
\usepackage[all]{xy}

\newcommand{\cc}{\ensuremath{\mathbb{C}}}
\newcommand{\dual}{\ensuremath{^{*}}}

\newcommand{\h}[1]{\-\mbox{-#1}}

\newcommand{\II}{\ensuremath{\mathcal{I}}}

\newcommand{\OO}{\ensuremath{\mathcal{O}}}

\newcommand{\pp}{\ensuremath{\mathbb{P}}}

\newcommand{\zz}{\ensuremath{\mathbb{Z}}}
\newcommand{\CC}{\ensuremath{\mathbb{C}}}

\newtheorem{te}{Theorem}[section]
\newtheorem*{te*}{Theorem}
\newtheorem*{tecat*}{Teorema}
\newtheorem{p}[te]{Proposition}
\newtheorem*{p*}{Proposition}
\newtheorem*{pcat*}{Proposició}

\newtheorem*{co*}{Corollary}
\newtheorem*{cocat*}{Coro{\lgem}ari}

\newtheorem{lem}[te]{Lemma}
\newtheorem*{lem*}{Lemma}

\theoremstyle{definition}

\newtheorem*{pr*}{Problem}
\newtheorem*{prcat*}{Problema}

\newtheorem{ob}[te]{Remark}
\newtheorem{que}[te]{Question}

\title[Splitting criteria on the isotropic Grassmannian]{Splitting criteria for vector bundles on the symplectic isotropic Grassmannian}



\author{Pedro Macias Marques$^{1}$ and Luke Oeding$^{2}$}

\address{$^{1}$Departamento de Matemática \\ Universidade de \'Evora \\ Rua Rom\~{a}o Ramalho, 59
\\ 7000--671 \'Evora, Portugal
}
\address{$^{1}$Departament d'\`Algebra i Geometria \\ Facultat de Matem\`atiques \\ Universitat de Barcelona \\ Gran Via de les Corts Catalanes, 585 \\ 08007 Barcelona, Espanya }
\email{pmm@uevora.pt}

\thanks{$^{1}$Partially supported by Funda\c{c}\~ao para a Ci\^encia e Tecnologia, under grant SFRH/BD/27929/2006, and by CIMA -- Centro de Investiga\c{c}\~ao  em Matem\'atica e Aplica\c{c}\~oes, Universidade de \'Evora.}

\address{$^{2}$Dipartimento di Matematica ``U. Dini'' \\ Universit\`a degli Studi di Firenze\\ Viale Morgagni 67/A\\ 50134 Firenze, Italy }
\email{oeding@math.unifi.it}
\thanks{$^{2}$Partially supported by NSF International Research Fellowship Program Grant \#08538000}

\date{2010/03/15}

\begin{document}

\maketitle

\begin{abstract}
We extend a theorem of Ottaviani on cohomological splitting criterion for vector bundles over the Grassmannian to the case of the symplectic isotropic Grassmanian. We find necessary and sufficient conditions for the case of the Grassmanian of symplectic isotropic lines.  For the general case the generalization of Ottaviani's conditions are sufficient for vector bundles over the symplectic isotropic Grassmannian. By a calculation in the program LiE, we find that Ottaviani's conditions are necessary for Lagrangian Grassmannian of isotropic $k$-planes for $k\leq 6$, but they fail to be necessary for the case of the Lagrangian Grassmannian of isotropic $7$-planes. Finally, we find a related set of necessary and sufficient splitting criteria for the Lagrangian Grassmannian.
\end{abstract}

\tableofcontents

\section{Introduction}

Splitting criteria for vector bundles dates from the sixties, when Horrocks presented the criterion for vector bundles over $\pp^n$ \cite{Hor64}. It is a very important tool to study classification of vector bundles, for instance. Some progress has been done since then. In 1981 Evans and Griffiths gave a slightly simpler criterion for vector bundles over $\pp^n$ with rank ${r\le n}$ \cite{EG81}. Ottaviani made a contibution to this problem in 1989 by presenting splitting criteria for vector bundles over Grassmannians and quadrics \cite{Ott89}. In 2000 Kumar and Rao gave a different criterion for rank~$2$ vector bundles over $\pp^n$, for ${n\ge4}$. In 2003 an improvement to Horrocks criteria was obtained by Kumar, Paterson and Rao \cite{KPR03} for vector bundles of rank ${r<n}$, if $n$ is even, and ${r<n-1}$, if $n$ is odd. In 2005 Costa and Mir\'o\h{Roig} extended Horrocks criterion to multiprojective spaces and to smooth projective varieties with the weak CM property \cite{CaMR05}. Malaspina recently generalized these reusults \cite{Mal08} and improved Ottaviani's result on quadrics \cite{Mal09}.

In this paper we make a contibution to extend Ottaviani's ideas to the symplectic isotropic Grassmannian $IG(k,n)$ (\emph{i.e.} the Grassmannian of projective $k$-planes in projective $n$-dimensional space which are isotropic for a non\h{degenerate} symplectic form, herein called the \emph{isotropic Grassmannian} or \emph{Lagrangian Grassmannian} when $n=2k+1$). In particular, we answer Question \ref{quest}, a version of which was posed at P.R.A.G.MAT.I.C. 2009, which, in short, was to attempt to generalize
Ottaviani's splitting conditions \cite{Ott89} to the case of the isotropic Grassmannian. Specifically, we give sufficient splitting conditions for a vector bundle over ${IG(k,n)}$ (Proposition~\ref{sufficient}). For the case of lines, we show that these sufficient conditions on ${IG(1,n)}$ are also necessary. For ${LG(k):=IG(k,2k+1)}$ and for the first instances of $k$, i.e.\ for ${1\le k\le6}$, we show that these conditions are necessary as well (Theorem~\ref{k6}). However, we show by a counter\h{example} that they fail to be necessary for ${k=7}$. This suggests that a different set of conditions must be explored, which is what we do in section~\ref{anotherset}, finding a splitting criterion for $LG(k)$.

In more detail, in Proposition~\ref{sufficient} we use Ottaviani's proof \cite[Theorem~2.1]{Ott89}, with a slight modification, explained in Remark~\ref{globsectquo}. Ottaviani states that a vector bundle on the Grassmannian $G(k,n)$ splits if and only if
\[H^i\left(G(k,n),\Lambda^{j_k}{Q}\dual\otimes\cdots\otimes
	\Lambda^{j_1}{Q}\dual\otimes E(t)\right)=0,\]
for all ${t\in\zz}$, ${0\le j_1,\ldots,j_k\le n-k}$, and ${0<i<(k+1)(n-k)=\dim\big(G(k,n)\big)}$, where~$Q$ is the quotient bundle on the Grassmannian. A naïve conjecture of a set of splitting conditions for the Lagrangian Grassmannian would be 
\[H^i\left(LG(k),\Lambda^{j_k}{Q}\dual\otimes\cdots\otimes
	\Lambda^{j_1}{Q}\dual\otimes E(t)\right)=0,\]
for all ${t\in\zz}$, ${0<i<\dim\big(LG(k)\big)}$ and ${0\le j_1,\ldots,j_k\le n-k}$,
where~$Q$ is the quotient bundle on the  Lagrangian Grassmannian. These conditions fail for $LG(2)$, because $H^1(LG(2),Q^*\otimes Q^*)$ contains $H^1(LG(2),\Omega^1)$ which is nonzero, and therefore a different set of conditions must be considered.
The idea to improve these condition is to relate $\sum j_k$ with the order $i$ of the cohomology groups.

The proof of Proposition~\ref{sufficient} goes by induction on~$k$ and uses a global section of the quotient bundle~$Q$ over ${IG(k,n)}$ and its Koszul complex to relate splitting conditions on ${IG(k,n)}$ to the ones on ${IG(k-1,n-2)}$. For the converse (with ${1\le k\le6}$), we decompose the bundle using the Pieri formula and then for each irreducible summand, we use Bott's algorithm and the Borel\h{Weyl}\h{Bott} theorem for computing cohomology of irreducible homogeneous vector bundles over homogeneous varieties. We used the computer program LiE to perform both algorithms (see Section \ref{cohomology}). Finally, in section~\ref{anotherset}, working with a different set of conditions, and using the methods of Proposition~\ref{sufficient} and of Section~\ref{cohomology}, we find a splitting criterion for~$LG(k)$.

\section{Splitting criterion on the isotropic Grassmannian}

Let $V$ be a complex vector space of dimension $n+1$, with $n$ odd, and let $\omega$ be a non\h{degenerate} symplectic form on~$V$. For each subspace $F$ of $V$ define
\[F^\bot:=\{v\in V:\omega(v,w)=0\text{ for all }w\in F\}.\]
Let $\pp^n:=\pp(V)$ and let $Gr(k,n)$ be the Grassmannian of \mbox{$k$-planes} in $\pp^n$. Consider the isotropic Grassmannian
\[IG(k,n):=\left\{F\in Gr(k,n):F\subseteq F^\bot\right\}.\]
For every $k\ge0$ and for every odd $n>2k$, this variety is non\h{empty} and we have 
\[\dim IG(k,n)=\tfrac{1}{2}(k+1)(2n-3k).\]
When ${n=2k+1}$, \mbox{$k$-planes} in $\pp^n$ correspond to half\h{dimensional} vector subspaces of~$V$, and isotropic subspaces are Lagrangian, so we call ${IG(k,2k+1)}$ the Lagrangian Grassmannian and denote it by ${LG(k)}$.

\begin{ob}\label{globsectquo}
Let $Q$ be the quotient bundle on $IG(k,n)$ and let $s$ % $\xymatrix{s:IG(k,n)\ar[r]&Q_{k,n}}$ 
be a non\h{zero} global section of $Q$. Then there is a vector ${v\in V}$ such that ${s(F)=(F,\overline{v})}$ for all ${F\in IG(k,n)}$, where $\overline{v}$ is the class of~$v$ in the fibre $\tfrac{V}{F}$. Therefore we get ${s(F)=0}$ if and only if ${v\in F}$. Since $\omega$ is non\h{degenerate}, there is a vector ${v'\in V}$ such that ${\omega\left(v,v'\right)\ne0}$. Now $\dim\langle v\rangle^\bot=\dim V-1=n$. Therefore $V$ admits a base $v,v_1,\ldots,v_{n-1},v'$, with $\langle v\rangle^\bot=\langle v,v_1,\ldots,v_{n-1}\rangle$. Note that since~$\omega$ is skew\h{symmetric}, ${\omega(v,v)=0}$, and therefore ${v\in\langle v\rangle^\bot}$. If $F$ is in the zero locus of~$s$, there is a subspace~$F'$ of~$V$ such that ${F=\langle v\rangle\oplus F'}$. Furthermore, since ${F\subseteq F^\bot}$, we get that $F'$ can be chosen as a subspace of $\langle v_1,\ldots,v_{n-1}\rangle$, and this gives us an isomorphism between ${IG(k-1,n-2)}$ and the zero locus~$Z$ of~$s$.

Finally, observe that a fiber of $Q$ in a point ${F\in Z}$ admits a decomposition
\[\frac{V}{F}=\frac{\langle v\rangle^\bot\oplus\langle v'\rangle}{\langle v\rangle\oplus F'}
	\cong\frac{\langle v\rangle^\bot}{\langle v\rangle\oplus F'}\oplus\langle v'\rangle,\]
where $F'$ and $v'$ are as above. From here, we get that
\[{{Q}_{|Z}\cong \tilde{Q}\oplus\OO_Z},\]
where $\tilde{Q}$ is the quotient bundle on ${IG(k-1,n-2)}$.
\end{ob}

%I think we need this statement too%
The following lemma is a consequence from the previous remark.
\begin{lem}\label{restriction}
Let $Q$ (respectively $\tilde{Q}$) be the quotient bundle on $IG(k,n)$ (respectively $Z = IG(k-1,n-2)$) as above. Then
\[
\left(\bigwedge^{p}Q\right)_{|Z} 
\cong
\left( \bigwedge^{p} \tilde{Q} \right) \oplus \left(\bigwedge^{p-1} \tilde{Q}\right)
,\]
for all $1\leq p \leq n-k$.%, and in particular when $p = n-k$,
%\[
%\OO_{IG(k,n)}(1) \cong \bigwedge^{n-k}Q  \cong \bigwedge^{n-k-1}\tilde Q \cong \OO_{Z}(1)
%,\]
\end{lem}
%\begin{proof}
%By Remark~\ref{globsectquo}, we know that
%\[{{Q}_{|Z}\cong \tilde{Q}\oplus\OO_Z}.\]
%Work on the fibers and restrict the representation $\bigwedge^{p}\frac V F$  of $Sp(2n)$ to a representation of the subgroup $Sp(n-1)\times Sp(1) \subset Sp(2n)$.  Then decompose this representation by the following general rule:  If $A$ and $B$ are respectively representations of $GL(n)$ and $GL(m)$, then
%\[
%\bigwedge^{p}(A \oplus B) 
%= \bigoplus_{i=0}^{p}
%\bigwedge^{p-i}A \otimes \bigwedge^{i}B
%.\]
%Since $\OO_{Z}$ is a line bundle, only the two terms corresponding to $i=0,1$ are nonzero when $p<n-k$, and only the term corresponding to $i=1$ is nonzero when $p = n-k$.
%\end{proof}
%
%\[\bigwedge^{j_{k-1}}{\tilde{Q}}\dual\otimes\cdots\otimes
%	\bigwedge^{j_{1}}{\tilde{Q}_{1}}\dual\otimes E_{|Z}(t)\]
%is a summand of
%\[\bigwedge^{j_{k-1}}{Q_{k-1}}\dual\otimes\cdots\otimes
%		\bigwedge^{j_{1}}{Q_{1}}\dual\otimes E(t)_{|Z}\]
%and hence the vanishing of the cohomology above implies
%\[H^i\left(Z,\bigwedge^{j_{k-1}}{\tilde{Q}_{k-1}}\dual\otimes\cdots\otimes
%	\bigwedge^{j_{1}}{\tilde{Q}_{1}}\dual\otimes E_{|Z}(t)\right)=0.\]

Because we will use it several times, we include for reference the following lemma which appears in \cite{Ott89}
\begin{lem}[Lemma 1.1(i) \cite{Ott89}]
Let
\[
\xymatrix{0\ar[r]&A_{n} \ar[r]&\ldots \ar[r]&A_{0}\ar[r]& B \ar[r]& 0}
\]
be an exact sequence of sheaves on a variety $X$, let $r$ be an integer $\geq 0$.
If $H^{r+i}(X,A_{i})=0$ for $i =0,\ldots,n$ then $H^{r}(X,B)=0$.
\end{lem}

In order to strengthen our splitting conditions, we will require the following.
Let $Q_{q}$ denote the quotient bundle on $IG(k,n)$, for each $1\leq q \leq k$.  There is only one quotient bundle on $IG(k,n)$, but we use the parameter $q$ as a placeholder so that we know when each bundle occurs in our proof. 

Let $\bigwedge^{j_{q}}Q_{q}$ denote the exterior power for $0\leq j_{q} \leq n-2k+q -1$. For notational convenience we will write $\bigwedge^{j_{q}}Q_{q}$ only as a placeholder in the case $j_{q} = n-2k+q$, but in this case, we actually replace $\bigwedge^{n-2k+q}Q_{q}$ with a line bundle.
The reason for this notation is that in our proof,  at the $q^{th}$ step of induction we will use the fact that when $rank(Q_{q}) = n-2k+q$, $\bigwedge^{n-2k+q}Q_{q}$ is a line bundle on $IG(q,n-2k+2q)$.  With this notation, our conditions are easier to state because the degree $i$ of cohomology always depends on the $j_{q}$'s by the same expression.  If we were to not use this notational convenience, we would have several different expressions for the ranges of the index $i$, each depending on the values of the $j_{q}$.

\begin{p}[Sufficient splitting criterion]\label{sufficient}
Let ${n\ge3}$ be an odd number and let $E$ be a vector bundle on the isotropic Grassmannian $IG(k,n)$ such that
\[H^i\left(IG(k,n),\bigwedge^{j_k}{Q_{k}}\dual\otimes\cdots\otimes
	\bigwedge^{j_1}{Q_{1}}\dual\otimes E(t)\right)=0,\]
for all $t\in\zz$, $i>0$ and $j_1,\ldots,j_k$ such that ${0\le j_q\le n-2k+q}$, (with the convention that $\bigwedge^{n-2k+q}Q_{q}$ is replaced by a line bundle for each $q$) and  
\[\sum_{q=1}^k j_q\le i<\sum_{q=1}^kj_q+n-2k.\]
Then $E$ splits as a sum of line bundles.
\end{p}

\begin{proof}
This proof is analogous to Ottaviani's proof for the regular Grassmannian case \cite{Ott89}. We proceed by induction on~$k$. For ${k=0}$, we have
\[{IG(0,n)\cong G(0,n)\cong\pp^n}.\]
Therefore the condition in the theorem amounts to saying that $E(t)$ has no intermediate cohomology for all $t\in\zz$. By Horrocks criterion \cite[chapter~I, Theorem~2.3.1]{OSS80}, $E$ splits.

Let ${k>0}$ and assume the proposition holds for ${k-1}$. Let $E$ be a vector bundle on ${IG(k,n)}$ satisfying the conditions of the proposition. Let $s$ be a global section of $Q_{k}$. By Remark~\ref{globsectquo}, its zero locus~$Z$ is isomorphic to ${IG(k-1,n-2)}$. We wish to use the hypothesis of induction on~$E_{|Z}$, and for that we will prove the vanishing
\[H^i\left(Z,\bigwedge^{j_{k-1}}{\tilde{Q}_{k-1}}\dual\otimes\cdots\otimes
	\bigwedge^{j_{1}}{\tilde{Q}_{1}}\dual\otimes E_{|Z}(t)\right)=0,\]
for the corresponding values of $i$, $j_1,\ldots,j_{k-1}$, where $\tilde{Q}$ denotes the quotient bundle on $IG(k-1,n-2)$.

Let ${t\in\zz}$ and let ${j_1,\ldots,j_{k-1}}$ be such that
\[0\le j_q\le n-2k+q,\]
for ${1\le q\le k-1}$. Denote
\[{A_j:=\bigwedge^{j-1}{Q_{k}}\dual\otimes\bigwedge^{j_{k-1}}{Q_{k-1}}\dual\otimes\cdots\otimes
	\bigwedge^{j_{1}}{Q_{1}}\dual\otimes E(t)},\]
for ${1\le j\le n-k+1}$. Consider the Koszul complex of~$s$
\begin{multline*}
\xymatrix{0\ar[r]&{\bigwedge^{n-k}{Q_{k}}\dual}\ar[r]&{\bigwedge^{n-k-1}{Q_{k}}\dual}
		\ar[r]&{\cdots}}\\
	\xymatrix{\cdots\ar[r]&\bigwedge^2{Q_{k}}\dual\ar[r]&{Q_{k}}\dual
		\ar[r]&\OO_{IG(k,n)}\ar[r]&\OO_Z\ar[r]&0.}
\end{multline*}
and tensor it by ${\bigwedge^{j_{k-1}}{Q_{k-1}}\dual\otimes\cdots\otimes\bigwedge^{j_{1}}{Q_{1}}\dual\otimes E(t)}$:
\begin{multline*}
\xymatrix{0\ar[r]&A_{n-k+1}\ar[r]&A_{n-k}\ar[r]&\cdots\ar[r]&A_{2}\ar[r]&}\\
	\xymatrix{\ar[r]&A_{1}\ar[r]&\bigwedge^{j_{k-1}}{Q_{k-1}}\dual\otimes\cdots\otimes\bigwedge^{j_{1}}{Q_{1}}\dual\otimes E(t)_{|Z}\ar[r]&0}
\end{multline*}

Let $i$ be such that
\[\sum_{q=1}^{k-1}j_q\le i<\sum_{q=1}^{k-1}j_q+(n-2)-2(k-1).\]
Note that ${(n-2)-2(k-1)=n-2k}$. We will apply Lemma~1.1 in \cite{Ott89} to this exact sequence. Assume ${1\le j\le n-k+1}$, and denote $j_{k} = j-1$.
%${j_1':=j-1}$, ${j_{q+1}':=j_q}$, for ${1\le q\le k-1}$. %not necessary since I changed the order of the bundles
By our hypothesis on~$E$, we get
\begin{multline*}
H^{i+j-1}\big(IG(k,n),A_j\big)=\\
	=H^{i+j_{k}}\left(IG(k,n),\bigwedge^{j_{k}}{Q_{k}}\dual\otimes \bigwedge^{j_{k-1}}{Q_{k-1}}\dual \otimes\cdots\otimes
		\bigwedge^{j_1}{Q_{1}}\dual\otimes E(t)\right)=0,
\end{multline*}
and since we have ${0\le j_k\le n-k}$ the bound on $i$ is
\[\sum_{q=1}^{k}j_q\le i+j_{k}<\sum_{q=1}^{k}j_q+n-2k.\]
Therefore we get
\begin{equation}\label{star}
H^i\left(Z,\bigwedge^{j_{k-1}}{Q_{k-1}}\dual\otimes\cdots\otimes
		\bigwedge^{j_{1}}{Q_{1}}\dual\otimes E(t)_{|Z}\right)=0.
		\end{equation}
%%% I think we should put the following in a separate lemma above
By Lemma \ref{restriction}, we know that if $Q$ is the quotient bundle on $IG(k,n)$ and $\tilde{Q}$ is a quotient bundle on $IG(k-1,n-2)$, then
\[\bigwedge^{j_{k-1}}{\tilde{Q}_{k-1}}\dual\otimes\cdots\otimes
	\bigwedge^{j_{1}}{\tilde{Q}_{1}}\dual\otimes E_{|Z}(t)\]
is a summand of
\[\bigwedge^{j_{k-1}}{Q_{k-1}}\dual\otimes\cdots\otimes
		\bigwedge^{j_{1}}{Q_{1}}\dual\otimes E(t)_{|Z}\]
and hence the vanishing of the cohomology above implies
\[H^i\left(Z,\bigwedge^{j_{k-1}}{\tilde{Q}_{k-1}}\dual\otimes\cdots\otimes
	\bigwedge^{j_{1}}{\tilde{Q}_{1}}\dual\otimes E_{|Z}(t)\right)=0.\]
%%% see above
	
By the induction hypothesis, $E_{|Z}$ splits. We can therefore consider a splitting bundle~$F$ on ${IG(k,n)}$ and an isomorphism $\xymatrix@1{\alpha_0:F_{|Z}\ar[r]&E_{|Z}}$, with \[{\alpha_0\in H^0\big(Z,(F\dual\otimes E)_{|Z}\big)}.\] We wish to extend this isomorphism to a morphism ${\alpha\in H^0\big(IG(k,n),F\dual\otimes E\big)}$. Now tensor the exact sequence
\[\xymatrix{0\ar[r]&\II_Z\ar[r]&\OO_{IG(k,n)}\ar[r]&\OO_Z\ar[r]&0}\]
by ${F\dual\otimes E}$ to get
\[\xymatrix{0\ar[r]&\II_Z\otimes F\dual\otimes E\ar[r]&F\dual\otimes E\ar[r]&(F\dual\otimes E)_{|Z}\ar[r]&0.}\]
By our hypothesis on~$E$ and~$F$, the large cohomology sequence gives us
\begin{multline*}
\xymatrix{0\ar[r]&H^0\left(\II_Z\otimes F\dual\otimes E\right)\ar[r]&}\\
	\xymatrix@C-2pt{H^0\left(F\dual\otimes E\right)\ar[r]&H^0\left((F\dual\otimes E)_{|Z}\right)
		\ar[r]&H^1\left(\II_Z\otimes F\dual\otimes E\right)\ar[r]&0.}
\end{multline*}
To show that $\alpha_{0}$ lifts to a morphism ${\alpha\in H^0\big(IG(k,n),F\dual\otimes E\big)}$, we will show that the map $\xymatrix@1{H^0\left(F\dual\otimes E\right)\ar[r]&H^0\left((F\dual\otimes E)_{|Z}\right)}$
is surjective by showing 
\[H^1\left(\II_Z\otimes F\dual\otimes E\right) =0.\]

Consider again the Koszul complex of~$s$, this time ending in $\II_Z$,
\begin{multline*}
\xymatrix{0\ar[r]&{\bigwedge^{n-k}{Q_{k}}\dual}\ar[r]&{\bigwedge^{n-k-1}{Q_{k}}\dual}
		\ar[r]&{\cdots}}\\
	\xymatrix{\cdots\ar[r]&\bigwedge^2{Q_{k}}\dual\ar[r]&{Q_{k}}\dual
		\ar[r]&\II_Z\ar[r]&0.}
\end{multline*}
and tensor it by ${F\dual\otimes E}$:
\begin{multline*}
\xymatrix{0\ar[r]&{\bigwedge^{n-k}{Q_{k}}\dual\otimes F\dual\otimes E}
	\ar[r]&{\bigwedge^{n-k-1}{Q_{k}}\dual\otimes F\dual\otimes E}
		\ar[r]&{\cdots}}\\
	\xymatrix{\cdots\ar[r]&{Q_{k}}\dual\otimes F\dual\otimes E
		\ar[r]&\II_Z\otimes F\dual\otimes E\ar[r]&0.}
\end{multline*}
Our hypotheses on~$E$ include the condition that %I have put "hypotheses", in plural
\[
H^{j_{k}}\left(\bigwedge^{j_{k}}{Q_{k}}\dual \otimes E(t)\right) =0
\]
so again using Lemma~1.1 in \cite{Ott89} we have
\[H^1\left(\II_Z\otimes F\dual\otimes E\right)=0\]
Therefore the morphism $\xymatrix@1{H^0\left(F\dual\otimes E\right)\ar[r]&H^0\left((F\dual\otimes E)_{|Z}\right)}$ is surjective, and we have ${\alpha\in H^0\big(IG(k,n),F\dual\otimes E\big)}$ such that ${\alpha_{|Z}=\alpha_0}$. 

Now consider $\xymatrix@1{\det\alpha:\det F\ar[r]&\det E}$, where 
%\begin{multline*}
%\det\alpha\in H^0\big(IG(k,n),(\det F)\dual\otimes\det E\big)\cong 
%	H^0\big(IG(k,n),(\det F)\dual\otimes\det E\big)
%\end{multline*}
\begin{multline*}
\det\alpha\in H^0\big((\det F)\dual\otimes\det E\big)\cong\\
	\cong H^0\big(\OO_{IG(k,n)}(c_1E-c_1F)\big)=H^0\left(\OO_{IG(k,n)}\right)\cong\cc.
\end{multline*}
We conclude that $\alpha$ is a constant. Since it is non\h{zero} on~$Z$, it is non\h{zero} on all ${IG(k,n)}$, and hence an isomorphism. 
%For ${j=n-k+1}$, we have ${\bigwedge^{j-1}{Q_{k,n}}\dual=\bigwedge^{n-k}{Q_{k,n}}\dual\cong\OO_IG(t')}$, for some ${t'\in\zz}$, and we get, again by our hypothesis on~$E$, 
%\begin{multline*}
%H^{i+j-1}(A_j)=\\
%	=H^{i+n-k}\left(\bigwedge^{n-k}{Q_{k,n}}\dual\otimes
%		\bigwedge^{j_1}{Q_{k,n}}\dual\otimes\cdots\otimes
%		\bigwedge^{j_{p}}{Q_{k,n}}\dual\otimes E(t)\right)\\
%	\cong H^{i+n-k}\left(\bigwedge^{j_1}{Q_{k,n}}\dual\otimes\cdots\otimes
%		\bigwedge^{j_{p}}{Q_{k,n}}\dual\otimes E\left(t+t'\right)\right)=0,
%\end{multline*}
%since
%\[\sum_{q=1}^{p}j_q<\sum_{q=1}^{p}j_q+n-k\le i+n-k\]
%and
%\begin{multline*}
%i+n-k<\sum_{q=1}^{p}j_q+\tfrac{1}{2}(k-p)(2n-3k-p-1)+n-k\\
%	\le
%\end{multline*}
%\begin{multline*}
%\xymatrix{0\ar[r]&{\bigwedge^{n-k}{Q_{k,n}}\dual\otimes
%		\bigwedge^{j_1}{Q_{k,n}}\dual\otimes\cdots\otimes\bigwedge^{j_p}{Q_{k,n}}\dual\otimes E(t)}
%		\ar[r]&{\bigwedge^{n-k-1}{Q_{k,n}}\dual\otimes
%		\bigwedge^{j_1}{Q_{k,n}}\dual\otimes\cdots\otimes\bigwedge^{j_p}{Q_{k,n}}\dual\otimes E(t)}
%		\ar[r]&{\cdots}}\\
%	\xymatrix{\cdots\ar[r]&{\bigwedge^2{Q_{k,n}}\dual\otimes
%		\bigwedge^{j_1}{Q_{k,n}}\dual\otimes\cdots\otimes\bigwedge^{j_p}{Q_{k,n}}\dual\otimes E(t)}
%		\ar[r]&{{Q_{k,n}}\dual\otimes
%		\bigwedge^{j_1}{Q_{k,n}}\dual\otimes\cdots\otimes\bigwedge^{j_p}{Q_{k,n}}\dual\otimes E(t)}
%		\ar[r]&		
%		{\bigwedge^{j_1}{Q_{k,n}}\dual\otimes\cdots\otimes\bigwedge^{j_p}{Q_{k,n}}\dual\otimes E(t)}
%		\ar[r]&\OO_Z\ar[r]&0.}
%\end{multline*}
\end{proof}

\begin{p}[Splitting criterion on $IG(1,n)$.]\label{lines}
Let ${n\ge3}$ be an odd number and let $E$ be a vector bundle on the isotropic Grassmannian $IG(1,n)$. Then $E$ splits as a sum of line bundles if and only if
%\[H^i\left(IG(k,n),\bigwedge^j{Q_{1,n}}\dual\otimes E(t)\right)=0,\]
\[H^i\left(IG(k,n),\bigwedge^jQ\dual\otimes E(t)\right)=0,\]
for all $t\in\zz$ and all $i,j$ such that ${0<i<\dim IG(1,n)}$ and ${0\le j<n-1}$.
\end{p}

\begin{proof}
Note that since %${\bigwedge^{n-1}{Q_{1,n}}\dual\cong\OO(t')}$
${\bigwedge^{n-1}Q\dual\cong\OO(t')}$ for some ${t'\in\zz}$, the vanishing required in Proposition~\ref{sufficient} is guaranteed. Therefore every vector bundle~$E$ on~$IG(1,n)$ satisfying the conditions of this proposition splits.

For the converse, note that $IG(1,n)$ is a hyperplane section of $Gr(1,n)$. Specifically, it is the hyperplane section given by $\omega = 0$. Therefore, we can consider the exact sequence
\[\xymatrix{0\ar[r]&\OO_{Gr(1,n)}(-1)\ar[r]&\OO_{Gr(1,n)}\ar[r]&
	\OO_{IG(1,n)}\ar[r]&0}\]
and twist it by ${\bigwedge^jQ\dual(t)}$ to get
%\[\xymatrix{0\ar[r]&\bigwedge^j{Q_{1,n}}\dual(t-1)\ar[r]&\bigwedge^j{Q_{1,n}}\dual(t)
%	\ar[r]&\left(\bigwedge^j{Q_{1,n}}\dual(t)\right)_{|IG(1,n)}\ar[r]&0.}\]
\[\xymatrix{0\ar[r]&\bigwedge^jQ\dual(t-1)\ar[r]&\bigwedge^jQ\dual(t)
	\ar[r]&\big(\bigwedge^jQ\dual(t)\big)_{|IG(1,n)}\ar[r]&0.}\]
For any ${0<i<\dim IG(1,n)}$, the long cohomology sequence associated to this small exact sequence has the following terms
%\[\xymatrix{\cdots\ar[r]&H^i\left(\bigwedge^j{Q_{1,n}}\dual(t)\right)
%	\ar[r]&H^i\left(\left(\bigwedge^j{Q_{1,n}}\dual(t)\right)_{|IG(1,n)}\right)
%	\ar[r]&H^{i+1}\left(\bigwedge^j{Q_{1,n}}\dual(t-1)\right)\ar[r]&\cdots}\]
\begin{multline*}
\xymatrix{\cdots\ar[r]&H^i\big(\bigwedge^jQ\dual(t)\big)
		\ar[r]&H^i\left(\big(\bigwedge^jQ\dual(t)\big)_{|IG(1,n)}\right)
		\ar[r]&}\\
	\xymatrix{\ar[r]&H^{i+1}\big(\bigwedge^jQ\dual(t-1)\big)\ar[r]&\cdots}
\end{multline*}
Since all intermediate cohomology of ${\bigwedge^jQ\dual(t)}$ on the Grassmannian vanishes \cite{Ott89}, we get ${H^i\big(\bigwedge^jQ\dual(t)\big)=H^{i+1}\big(\bigwedge^jQ\dual(t-1)\big)=0}$, we get 
\[{H^i\left(\left(\bigwedge^jQ\dual(t)\right)_{|IG(1,n)}\right)=0}\]
\end{proof}

Now we focus our attention on the Lagrangian Grassmannian, where $n = 2k+1$.
The index ranges for the sufficient conditions in Theorem \ref{sufficient} are all $i,j_{1},\dots,j_{k}$ such that
${0\le j_q\le n-2k+q}$ and
\[\sum_{q=1}^kj_q\le i<\sum_{q=1}^kj_q+n-2k.\]
When $n = 2k+1$ we only have
${0\le j_q\le q+1}$ and
\[\sum_{q=1}^kj_q\le i<\sum_{q=1}^kj_q+1, \Rightarrow i=\sum_{q=1}^kj_q\]
We ask if these sufficient conditions are also necessary:

\begin{que}\label{quest}
Prove or disprove:
Let $k\geq 1$ and let $E$ be a vector bundle on the Lagrangian Grassmanian $LG(k)$. 
Let $Q_{q}$ denote the quotient bundle on $LG(k)$, for each $1\leq q \leq k$, and let $\bigwedge^{j_{q}}Q_{q}$ denote the exterior power for $0\leq j_{q} \leq q$. 
For notational convenience we use $\bigwedge^{j_{q}}Q_{q}$ as a placeholder in the case $j_{q} = q+1$, but in this case, we replace $\bigwedge^{q+1}Q_{q}$ with a line bundle.

Then $E$ splits as a sum of line bundles if and only if
\[H^i\left(\bigwedge^{j_{k}}{Q_{k}}\dual\otimes\cdots
	\otimes\bigwedge^{j_1}{Q_{1}}\dual\otimes E(t)\right)=0,\]
for all $t\in\zz$ and all $i$, $j_1, \ldots, j_k$ such that ${0\le j_q\le q+1}$ and
\[i=\sum_{q=1}^kj_q.\]
\end{que}

In order to verify cases of Question \ref{quest}, we assume $E$ splits and we need to calculate cohomology of the bundles  $\bigwedge^{j_1} Q^* \otimes \dots \otimes \bigwedge^{j_k}Q^*$.  For this we need to do two standard calculations.  First we decompose the bundle using the Pieri formula (see \cite{Fulton-Harris} for a detailed account).  Then for each decomposable summand, we use Bott's algorithm and the Borel-Weil-Bott Theorem (see \cite{Baston-Eastwood} or \cite{Fulton-Harris} for a detailed account) for computing cohomology of irreducible homogeneous vector bundles over homogeneous varieties.  These algorithms are both straightforward to perform in the computer program LiE.  Our code may be obtained by contacting the authors.  Here we state the results of these computations, while a more detailed account can be found in Section \ref{cohomology}

We found that Question \ref{quest} is valid for $k=1\dots 6$, but found several counterexamples for $k=7$. One such counterexample is 
\[
 H^{24}\left(L(7),\bigwedge^6 Q^* \otimes \bigwedge^5 Q^* \otimes \bigwedge^4 Q^* \otimes \bigwedge^3 Q^* \otimes \bigwedge^3 Q^* \otimes \bigwedge^2 Q^* \otimes Q^*(-9)\right) = \CC
 \]
 thus violating the conditions of Question \ref{quest}. 
We believe that Theorem \ref{sufficient} may be improved so that a finer version of Question \ref{quest} could be valid, and leave this for further study.

For completeness, we state the following
\begin{te}\label{k6}
Let $1\leq k \leq 6$ and let $E$ be a vector bundle on the Lagrangian Grassmanian $LG(k)$.
Let $Q_{q}$ denote the quotient bundle on $LG(k)$, for each $1\leq q \leq k$, and let $\bigwedge^{j_{q}}Q_{q}$ denote the exterior power for $0\leq j_{q} \leq q$. 
For notational convenience we use $\bigwedge^{j_{q}}Q_{q}$ as a placeholder in the case $j_{q} = q+1$, but in this case, we replace $\bigwedge^{q+1}Q_{q}$ with a line bundle.

Then $E$ splits as a sum of line bundles if and only if
\[H^i\left(\bigwedge^{j_{k}}{Q_{k}}\dual\otimes\cdots
	\otimes\bigwedge^{j_1}{Q_{1}}\dual\otimes E(t)\right)=0,\]
for all $t\in\zz$ and all $i,j_1 , \ldots ,j_k$ such that ${0\le j_q\le q+1}$ and
\[i=\sum_{q=1}^kj_q.\]
\end{te}

\section{Cohomology of the isotropic Grassmannian}\label{cohomology}

This section is aimed at the reader who may not be familiar with the program LiE and its use for Lie algebra calculations.  Our goal is to give an idea of how we carried out our calculations which verify the cases of Question \ref{quest} leading to Theorem \ref{k6}.  These same calculations showed that in the case $k=7$ there is a counterexample to Question \ref{quest}.

Having determined sufficient splitting conditions for vector bundles over the isotropic Grassmannian in Proposition \ref{sufficient}, we need to check whether the required vanishing of cohomology actually occurs.  In this case we assume that $E\rightarrow IG(k,n)$ splits as a direct sum of line bundles.  Because cohomology is additive, we may assume that $E$ is a line bundle and (by changing the twist by $t$ if necessary) that the vector bundle 
\[\bigwedge^{j_k}{Q}\dual\otimes
\cdots \otimes\bigwedge^{j_1}{Q}\dual\otimes E(t)\]
is isomorphic to the homogeneous (non-reduced) vector bundle 
\[\bigwedge^{j_k}{Q}\dual\otimes
\cdots \otimes\bigwedge^{j_{1}}{Q}\dual(t).\]

A variety $X$ is said to be a homogeneous variety if it is of the form $X = G/P$ for $G$ a simply connected complex semisimple Lie group and $P$ a parabolic subgroup.
The isotropic Grassmannian $IG(k,n)$ is a homogeneous variety for the symplectic group ${Sp(n+1)}$. As mentioned above, the vector bundles we are reduced to studying are homogeneous vector bundles.
The main tool available to calculate cohomology of irreducible homogeneous vector bundles over homogeneous varieties is the theorem of Borel-Weil and the algorithm given by Bott's theorem. In order to state this theorem (see Section \ref{Bott}), we need to recall a bit of representation theory, which can be found in \cite{Fulton-Harris,Ott89} for example.  But before we can even use the Borel-Weil-Bott theorem, we need to decompose each vector bundle into its irreducible components.  In general we would use the Littlewood-Richardson formula, but because we are only dealing with wedge powers of the dual of the quotient bundle, we can use the simpler Pieri formula.  We discuss this computation in Section \ref{Pieri}.

For small examples, both the Bott algorithm and Pieri formula are easy to execute by hand, but in order to gather evidence for Question \ref{quest}, we automated our calculations in the (free) computational package LiE.  We discuss this computation in Section \ref{LiE}.

\subsection{A sketch of Bott's algorithm} \label{Bott}

Instead of trying to repeat a course on representation theory, we record the practical definitions of the objects we use, and refer the reader to the literature for the general case, see for example \cite{Baston-Eastwood}. Let $G$ be a simply connected complex semisimple Lie group and let $P$ be a parabolic subgroup so that $G/P$ is a rational homogeneous variety. A key point is the following fact: The category of homogeneous vector bundles $E$ over a homogeneous variety $G/P$ is equivalent to the category of $P$-modules.  Irreducible $P$-modules are indexed by discrete data, and this is the set of data we use for our computations. 

Here is a sketch of Bott's algorithm. Start with the data of a fixed semi-simple Lie group $G$ and parabolic subgroup $P$.  The input is a string of integers $w$ called a weight representing an irreducible vector bundle over $G/P$.  Bott's algorithm outputs either a weight representing the cohomology and the degree in which that cohomology occurs (note that the Borel-Weil theorem implies that cohomology of irreducible vector bundles only occurs in one degree) or it outputs $0$ if the cohomology is singular, \emph{i.e.} if the cohomology vanishes in all degrees. 

The execution of Bott's algorithm goes as follows.  The data of $G$ and $P$ has attached to it a set of integer vectors $R_{+}$ called positive roots and an inner product $\langle,\rangle$ that allows one to pair the roots with weights as well as a group of reflections $\mathcal{W}_{G}$ called the Weyl group which acts on the weights by reflection.  

The statement of the Borel-Weil-Bott theorem uses the \emph{affine} action of the Weyl group:  First consider the distinguished weight vector $\rho = (1,\dots,1)$.  Then the affine action is defined as $w.\lambda := w(\lambda+\rho)-\rho$.

Step 1: Compute the pairing $\langle w,\alpha \rangle$ for all $\alpha \in R_{+}$.
	If $\langle w,\alpha \rangle = 0$ for any $\alpha \in R_{+}$, the cohomology is singular.
	Otherwise, the number $d$ of $\alpha$'s in $R_{+}$ such that $\langle w,\alpha \rangle$ is negative is the degree in which the (non-zero) cohomology occurs.
	
Step 2: In the non-singualr case, there is an element $\omega \in \mathcal{W}_{G}$ (determined by Bott's algorithm) which is the product of $d$ simple reflections (generators of $\mathcal{W}_{G}$) and the output $\omega.w$ is the output cohomology.

For our purposes we only need the information from Step 1, however the way that we implemented Bott's algorithm in LiE it actually computes Step 2 and as a result also outputs the information for Step 1.  When working by hand, Step 1 is often easier to execute.
\subsection{The Borel-Weil-Bott Theorem}
Following is more detail about the specific objects we use in our computations.
The Lagrangian Grassmanninan $LG(k)$ is a homogeneous variety of the form $G/P_{k+1}$, with $G = Sp(2(k+1))$ and $P_{k+1}$ a maximal parabolic. One reason to focus on the Lagrangian case is because the reductive part of $P_{k+1}$ is $SL(k+1)$ so we can decompose the $P_{k+1}$-modules using the representation theory of $SL(k+1)$-modules and the representation theory of $SL(k+1)$ is easier to deal with.

The irreducible homogeneous vector bundles over $LG(k)$ and hence the irreducible $P_{k+1}$-modules are indexed by strings of integers (called \emph{weights}) of the form $\lambda = (\lambda_{1},\dots,\lambda_{k+1})$, where only $\lambda_{k+1}$ is allowed to be negative.

A weight is called \emph{$G$-dominant} if $\lambda_{i} \geq 0$ for all $i$. In the case of $LG(k)$, a weight is called $P_{k+1}$-dominant if $\lambda_{k+1}$ is any integer, and the rest of $\lambda_{i}$ for $i\neq k+1$ are non-negative integers.  Bott's algorithm takes an input of a $P$-dominant weight and (if the cohomology is non-singular) outputs a $G$-dominant weight.

The \emph{simple roots} of $Sp(2(k+1))$ are associated to the following  length $k+1$ strings of integers
\[\begin{array}{cccc}
\alpha_{1} = (1,0\dots,0) &
\alpha_{2} = (0,1,\dots,0) &
\ldots &
\alpha_{k+1} = (0,0\dots,1)
\end{array}
\]
Let $1\leq i \leq k+1$. The positive roots of $Sp(2(k+1))$ are 
\[
\begin{array}{lr}
\alpha_{i} & 2 \alpha_{k}+ \alpha_{k+1} \\
\alpha_{i} + \alpha_{i+1} & \alpha_{k-1} + 2 \alpha_{k}+ \alpha_{k+1} \\
\alpha_{i} + \alpha_{i+1} + \alpha_{i+2} & 2\alpha_{k-1} + 2 \alpha_{k}+ \alpha_{k+1} \\
& \alpha_{k-2}+ 2\alpha_{k-1} + 2 \alpha_{k}+ \alpha_{k+1} \\
\vdots & \vdots \\
\alpha_{1}+ \dots + \alpha_{k+1} & 2\alpha_{1} + \dots + 2\alpha_{k}+ \alpha_{k+1}
\end{array}
\]

The positive roots of $Sp(2(k+1))$ whose associated reflections can move a $P$-dominant weight closer to being $G$-dominant are those with a positive integer in the $(k+1)^{st}$ spot.
 A weight $\lambda = (\lambda_{1},\dots,\lambda_{k+1})$ and a root $a = (a_{1},\dots,a_{k+1})$ pair as
\[
\langle\lambda,a\rangle = \sum_{i=1}^{k} \lambda_{i}a_{i} + 2\lambda_{k+1}a_{k+1}
\]
 A weight $\lambda$ is called \emph{singular} if
$\langle \lambda , \alpha \rangle = 0$,
for some positive root $\alpha$, otherwise $\lambda$ is called \emph{non-singular}.

The reflections in the hyperplanes perpendicular to the roots of the Lie algebra form the \emph{Weyl group} $\mathcal{W}$. 
The Weyl group is generated by the \emph{simple reflections} (reflections perpendicular to the simple roots). 
  For a given element $w \in \mathcal{W}$, the \emph{length} of $w$, $l(w)$, is the minimum number of simple reflections over all expressions of $w$. 

 In the case that $\lambda$ is non-singular, one checks that a shortest $w\in \mathcal{W}$ such that $w(\lambda)$ is $G$-dominant is such that $\lambda(w)$ is also the number of positive roots which pair with $\lambda$ to give a negative value.

\begin{te}[Borel-Weil-Bott]
Let $G$ be a simply connected complex semi\-simple Lie group and $P\subset G$ a parabolic subgroup.
Let $Q^{\lambda}$ be a homogeneous vector bundle over $G/P$ associated to the $P$-module of highest weight $\lambda$.  If $\lambda$ is singular, then
\[
H^{i}(G/P,Q^{\lambda}) = 0 \;\; \text{for all }i
\]
If $\lambda$ is non-singular, let $w$ be a shortest word in the Weyl group $\mathcal{W}$ so that $w.\lambda$ is $G$-dominant. Then
\begin{align*}
H^{l(w)}(G/P,Q^{\lambda}) &= \Gamma^{w.\lambda} \\
H^{i}(G/P,Q^{\lambda}) &= 0 \;\; \text{for all }i \neq l(w)
,\end{align*}
where $\Gamma^{w.\lambda}$ is the $G$-module of highest weight $w.\lambda$, and $w.\lambda$ is the affine action of $w\in \mathcal{W}$ on $\lambda$.
\end{te}

This theorem is implemented via Bott's algorithm in the program LiE. In light of Proposition \ref{sufficient},  we need to consider the case that the vector bundle $E$ splits and it remains to show the vanishing of 
\[H^i\left(\bigwedge^{j_k}{Q_k}\dual\otimes\cdots \otimes\bigwedge^{j_1}{Q_{1}}\dual(t)\right),\] 
for the appropriate index ranges stated in the theorem.
\subsection{Pieri's decomposition formula}\label{Pieri}

Bott's algorithm deals with irreducible vector bundles, but the vector bundles in Question \ref{quest} are not in general irreducible.  We need to decompose each bundle into its irreducible pieces and apply the Bott algorithm to those pieces.

In the case of the Lagrangian Grassmannian $LG(k)$, the reductive part of the parabolic $P_{k+1}$ is $SL(k+1)$. Therefore the quotient bundle and its exterior powers can be associated to representations of $SL(k+1)$, so from a representation theory standpoint, they are easier to deal with.

The vector bundles in the statement of Question \ref{quest} are all of the form 
$\bigwedge^{j_k}{Q}\dual\otimes\cdots 
\otimes\bigwedge^{j_1}{Q}\dual\otimes E(t)$.  Since we are doing calculations in the case that the vector bundle $E$ splits, we can just consider the vector bundle
$\bigwedge^{j_k}{Q}\dual\otimes\cdots 
\otimes\bigwedge^{j_1}{Q}\dual(t)$ -- note that this vector bundle is homogeneous and (in general) decomposable.

As mentioned in the beginning of this section, we need to decompose
the vector bundles of the form
$\bigwedge^{j_k}{Q}\dual\otimes\cdots 
\otimes\bigwedge^{j_1}{Q}\dual(t)$ into irreducible components.  We can accomplish this by decomposing the associated $P$-modules.  And (as mentioned above) because we are specializing to the Lagrangian Grassmannian case, we can work with representations of $SL(k+1)$.
  Let $F$ be the $P$-module associated to $Q\dual$. 
  
  Recall the Pieri formula for decomposing the tensor product of a representation $F^{\pi}$ indexed by the partition $\pi$ and $\bigwedge^{j}F$,
  \[
  F^{\pi} \otimes \bigwedge^{j}F = \bigoplus_{\lambda \sim}F^{\lambda}
  ,\]
  where $\lambda\sim$ is to indicate that the partitions $\lambda$ are constructed as Young diagrams from the Young diagram of $\pi$ by adding $j$ boxes, no two in the same row.
   We can apply the Pieri formula iteratively to decompose the tensor product:
\[
(\bigwedge^{k+1}F)^{\otimes t} \otimes
\bigwedge^{j_k}{F}\otimes\cdots 
\otimes\bigwedge^{j_1}{F} = \bigoplus_{\lambda\in B} F^{\lambda_{1},\dots,\lambda_{k+1}}
,\] 
where the condition $\lambda \in B$ means that $\lambda = (\lambda_{1},\dots, \lambda_{k+1})$ is a partition which can be constructed iteratively (via the Pieri formula) from the partitions $1^{j_{1}},\dots , 1^{j_{k}}$ and $t$ copies of the partition $(1^{k+1})$, where the notation $1^{p}$ indicates the partition $(1,\dots,1)$ with $1$ repeated $p$ times. In particular, the representation $F^{(j_{1},\dots,j_{k})'}$ occurs in the decomposition, where $\lambda'$ is the conjugate partition to $\lambda$.
This immediately implies that the cohomology 
\[H^{d}(IG(k,2k+1),\bigwedge^{j_k}{Q}\dual\otimes\cdots \otimes\bigwedge^{j_1}{Q}\dual(t))\]
 does not vanish if 
 \[H^{d}(IG(k,2k+1),(Q^{(j_{1},\dots,j_{k})'})\dual(t))\]
  does not vanish, where $(Q^{(j_{1},\dots,j_{k})'})\dual$ is the irreducible vector bundle associated to the irreducible $P$-module $F^{(j_{1},\dots,j_{k})'}$.

Notice that since $F$ has dimension $k+1$, we have an isomorphism %$F^{\lambda}\otimes \bigwedge^{k+1}F \simeq F^{\lambda}$, and if we keep track of the twists, we have 
$F^{\lambda}(t)\otimes \bigwedge^{k+1}F \simeq F^{\lambda}(t+1)$.  Since we are going to require vanishing for all twists by line bundles, we can just focus on the irreducible modules that occur in the decomposition up to isomorphism, and then consider all twists afterwards.

The index ranges that we need to consider are all $j_{1},\dots,j_{k}$ such that
$0\le j_q\le k-q+2$, and we need to consider cohomology which occurs in degree $d = \sum_{q=1}^kj_q$.
%
%A further consequence of the Pieri formula and the fact that we are considering all twists, we claim that to verify vanishing of cohomology of 
%\[H^{d}(IG(k,n),\bigwedge^{j_1}{Q}\dual\otimes\cdots \otimes\bigwedge^{j_k}{Q}\dual(t))\]
%for $ \sum_{q=1}^{k}j_{q} \leq d < \sum_{q=1}^{k}j_{q} +n-k$, and $0\leq j_{q} \leq q+1$,
%it is equivalent to check vanishing of cohomology for 
% \[H^{d}(IG(k,2k+1),(Q^{\lambda})\dual(t))\]
%for all $t\in\zz$ and all $\lambda = (\lambda_{1},\dots,\lambda_{s})$ such that
%${0\le \lambda_{q}\le k+1}$ for all $1\leq q \leq k$ and $s \leq k$.
%
However, we also need to consider two possible alternatives which could force us to consider cohomology in degree $d$ where $d \neq \sum_{q=1}^kj_q$.  One, that the decomposition of $\bigwedge^{j_{k}}Q^{*}\otimes\dots\otimes \bigwedge^{1}Q$ contains a representation indexed by a partition $\lambda$ that has at least $k+1$ parts. Two, that the bundle $\bigwedge^{j_k}{Q}\dual\otimes\cdots \otimes\bigwedge^{j_1}{Q}\dual(t)$ has least one $j_{q}= q+1$, in which case we would have replaced $\bigwedge^{j_q}{Q}\dual(t)$ by $\OO(t')$ for some other integer $t'$.

The first case is already handled by our script because of the following: LiE accepts the partition $\lambda$ and converts it to a weight vector $wt(\lambda)$.  If $\lambda$ has $k+1$ parts, then the $k+1^{st}$ entry in the vector $wt(\lambda)$ will be nonzero.  When we twist the bundle $Q^{\lambda}$ by $\OO(t)$, this is adds $t$ to the $k+1^{st}$ entry in $wt(\lambda)$, and this is no different if the $k+1^{st}$ entry in $wt(\lambda)$ is zero or nonzero.

In the second case, suppose we want to verify the vanishing of 
\[H^{d}\left(LG(k),\bigwedge^{j_k}{Q}\dual\otimes\cdots \otimes\bigwedge^{j_1}{Q}\dual(t)\right)\]
with $d = \sum_{q=1}^{k}j_{q}$
in the case that $j_{q} = q+1$ for some $q$.  This means that we need to verify for the same $d$, the vanishing of 
\[H^{d}\left(LG(k),\bigwedge^{j_k}{Q}\dual\otimes\cdots \otimes \widehat{\bigwedge^{j_q}{Q}\dual}\otimes\cdots \otimes\bigwedge^{j_1}{Q}\dual(t)\right),\]
where $\widehat{\bigwedge^{j_q}{Q}\dual}$ indicates omission.  We handle this case with an ``if'' statement at the last step of each loop in our script.

In the next subsection we describe our LiE scripts which test Question \ref{quest} leading to Theorem \ref{k6} in the cases $k=1 \ldots 6$ and provide our counter examples in the case $k=7$.

\subsection{LiE implementations}\label{LiE}
LiE \cite{LiE, LiE1992} is a computational package that allows us to compute the cohomology of the vector bundles in which we are interested.  In particular, we implement Bott's Algorithm to compute cohomology on vector bundles we constructed via iterative uses the Littlewood-Richardson or Pieri rule.  Herein we describe the scripts we wrote to accomplish these tasks.

The ``test'' script takes a partition and outputs the possible cohomology for each possible twist that could yield cohomology, printing a warning if there is any cohomology in the forbidden degree.  For each new $k$, we have to change $k$ and the group that LiE uses as default by hand.  The script tests each partition for intermediate cohomology and outputs the possible degrees for non-zero cohomology.

Here is our script for the case $k=3$, and the Lagrangian Grassmannian $LG(3)$.
\begin{verbatim}
test(vec w){
v = from_part(w);
k=3;
degrees=null(0);
setdefault(C4);
rho = all_one(k+1);
anything = 0;
for t=-2 to  3*(k+1) do
 CH=dominant(v+null(k)^[-t]+rho)-rho;
 myword = W_word(v+null(k)^[-t]+rho);
 ll = length(myword);
 mybool=0;
  for j = 1 to k+1 do 
   if CH[j] == -1 then mybool = 1; fi;
  od;
 if mybool ==0   && ll !=0 && ll != (k+1)*(k+2)/2 then 
degrees = degrees^[ll];
 fi;
 od;
degrees
}
\end{verbatim}

Next we have a script that decomposes each vector bundle and feeds the script ``test'' each irreducible component.  This uses the Littlewood-Richardson rule implemented in LiE.  We have included the case that whenever an index $j_{q} = q+1$ then we set the corresponding representation equal to the trivial representation - this is equivalent to removing that factor and replacing it with a line bundle.  Below is the example when $k=3$.
\begin{verbatim}
m=5
sum(vec v) = v*all_one(size(v))
for j1=0 to 2 do for j2=0 to 3 do for j3 =0 to 4 do
 v1 = all_one(j1)^null(m-j1); if j1==2 then v1 = null(m) fi; 
 v2 = all_one(j2)^null(m-j2); if j2==3 then v2 = null(m) fi;
 v3 = all_one(j3)^null(m-j3); if j3==4 then v3 = null(m) fi;
 t = LR_tensor(X v2,X v1); 
 t = LR_tensor(X v3,t);
 for i=1 to length(t) do
  degs = test(expon(t,i));
  for j = 1 to size(degs) do
  if degs[j] == sum([j3,j2,j1]) then 
   print("WE HAVE A PROBLEM"); print(t[i]);
   print("has cohomology in degree(s)");print(degs[j]); 
  else 
   print("ALL CLEAR"); fi;
  od;	
 od;
od;od;od
\end{verbatim}
We ran the same script, modified for the next cases, and found cohomology in a degree forbidden by the sufficient conditions of Question \ref{quest} for the first time at $k=7$.
The following vector bundles have non-zero cohomology in degree $24$, which was required to be zero in the question.
\[\bigwedge^6 Q^* \otimes  \bigwedge^5 Q^* \otimes  \bigwedge^4 Q^* \otimes  \bigwedge^3 Q^* \otimes  \bigwedge^3 Q^* \otimes  \bigwedge^2 Q^* \otimes  \bigwedge^1 Q^* \supseteq
     Q^{7,6,5,3,2,1,0}\]
\[\bigwedge^5 Q^* \otimes  \bigwedge^5 Q^* \otimes  \bigwedge^5 Q^* \otimes  \bigwedge^3 Q^* \otimes  \bigwedge^3 Q^* \otimes  \bigwedge^2 Q^* \otimes  \bigwedge^1 Q^* \supseteq
     Q^{6,6,6,3,3,0,0}\]
\[\bigwedge^5 Q^* \otimes  \bigwedge^5 Q^* \otimes  \bigwedge^5 Q^* \otimes  \bigwedge^3 Q^* \otimes  \bigwedge^3 Q^* \otimes  \bigwedge^2 Q^* \otimes  \bigwedge^1 Q^* \supseteq
     Q^{7,6,5,3,2,1,0}\]
\[\bigwedge^6 Q^* \otimes  \bigwedge^4 Q^* \otimes  \bigwedge^4 Q^* \otimes  \bigwedge^4 Q^* \otimes  \bigwedge^3 Q^* \otimes  \bigwedge^2 Q^* \otimes  \bigwedge^1 Q^* \supseteq
     Q^{7,5,5,5,1,1,0}\]
\[\bigwedge^6 Q^* \otimes  \bigwedge^4 Q^* \otimes  \bigwedge^4 Q^* \otimes  \bigwedge^4 Q^* \otimes  \bigwedge^3 Q^* \otimes  \bigwedge^2 Q^* \otimes  \bigwedge^1 Q^* \supseteq
     Q^{7,6,5,3,2,1,0}\]
\[\bigwedge^5 Q^* \otimes  \bigwedge^5 Q^* \otimes  \bigwedge^4 Q^* \otimes  \bigwedge^4 Q^* \otimes  \bigwedge^3 Q^* \otimes  \bigwedge^2 Q^* \otimes  \bigwedge^1 Q^* \supseteq
     Q^{6,6,5,5,2,0,0}\]
\[\bigwedge^5 Q^* \otimes  \bigwedge^5 Q^* \otimes  \bigwedge^4 Q^* \otimes  \bigwedge^4 Q^* \otimes  \bigwedge^3 Q^* \otimes  \bigwedge^2 Q^* \otimes  \bigwedge^1 Q^* \supseteq
     Q^{6,6,6,3,3,0,0}\]
\[\bigwedge^5 Q^* \otimes  \bigwedge^5 Q^* \otimes  \bigwedge^4 Q^* \otimes  \bigwedge^4 Q^* \otimes  \bigwedge^3 Q^* \otimes  \bigwedge^2 Q^* \otimes  \bigwedge^1 Q^* \supseteq
     Q^{7,5,5,5,1,1,0}\]
\[\bigwedge^5 Q^* \otimes  \bigwedge^5 Q^* \otimes  \bigwedge^4 Q^* \otimes  \bigwedge^4 Q^* \otimes  \bigwedge^3 Q^* \otimes  \bigwedge^2 Q^* \otimes  \bigwedge^1 Q^* \supseteq
     Q^{7,6,5,3,2,1,0}\]
  
%We might be able to dispense with these vector bundles by shuffling the order and replacing a $\bigwedge^{k+1}Q$ by a line bundle.  Then, instead of requiring the vanishing for the original bundle, we replace one factor with a line bundle, and we claim that this bundle has no cohomology in degree 24. Therefore none of these bundles violate our conditions.

%For example, reshuffle
%\begin{multline*}
%\bigwedge^6 Q^* \otimes  \bigwedge^5 Q^* \otimes  \bigwedge^4 Q^* \otimes  \bigwedge^3 Q^* \otimes  \bigwedge^3 Q^* \otimes  \bigwedge^2 Q^* \otimes  \bigwedge^1 Q^*
%\\
%\bigwedge^6 Q^* \otimes  \bigwedge^5 Q^* \otimes  \bigwedge^4 Q^* \otimes  \bigwedge^3 Q^* \otimes  \bigwedge^2 Q^* \otimes  \bigwedge^3 Q^* \otimes  \bigwedge^1 Q^*
%\end{multline*}
%and replace this bundle by 
%\[
% \bigwedge^6 Q^* \otimes  \bigwedge^5 Q^* \otimes  \bigwedge^4 Q^* \otimes  \bigwedge^3 Q^* \otimes  \bigwedge^2 Q^* \otimes  \bigwedge^1 Q^* (t')\]
%Then we verify that this bundle has no cohomology in degree $24$.

%Can I make this systematic?

\section{Another set of splitting conditions}\label{anotherset}
Instead of trying to prove splitting conditions by verifying complicated cohomology conditions at each stage, suppose we restrict our given bundle to $\pp^{1}$ in every case.  We use this idea to arrive at the following set of equivalent conditions for a vector bundle over the Lagrangian Grassmannian to be a splitting bundle.

\begin{te}\label{another}
Let $E$ be a vector bundle on the Lagrangian Grassmanian $LG(k)$ with $k\geq 1$, Let $Q^{(k)}$ denote the quotient bundle on $LG(k)$,  and let $\II_{LG(k-1)}$ denote the ideal sheaf associated to the tautological sequence 
\[
\xymatrix{0\ar[r]&\II_{LG(k-1)} \ar[r]& \OO_{LG(k)}\ar[r]& \OO_{LG(k-1)}\ar[r]& 0}
\]
%Let $Q$ denote the quotient bundle on $LG(k)$, $\bigwedge^{j}Q_{q}$ denote the exterior power for $0\leq j \leq k$.
The following are equivalent:
\begin{enumerate}
\item $E$ splits as a sum of line bundles.
\item There exists a a chain of smooth subvarieties $LG(0) \subset LG(1) \subset\ldots \subset LG(k)$ such that
%\[H^1\left(L(i),\II_{LG(i-1)} \otimes E_{|LG(i)}(t)\right)=0,\]
\[H^j\left(L(i),\bigwedge^{j}(Q^{(i)})^{*} \otimes E_{|LG(i)}(t)\right)=0,\]
for all $t\in\zz$ and all $1\leq i \leq k$ and for all $1\leq j \leq i+1$.
\item For every chain of smooth subvarieties $LG(0) \subset LG(1) \subset\ldots \subset LG(k)$ we have
%\[H^1\left(L(i),\II_{LG(i-1)} \otimes E_{|LG(i)}(t)\right)=0,\]
\[H^j\left(L(i),\bigwedge^{j}(Q^{(i)})^{*} \otimes E_{|LG(i)}(t)\right)=0,\]
for all $t\in\zz$ and all $1\leq i \leq k$ and for all $1\leq j \leq i+1$.
\end{enumerate}
\end{te}
\begin{proof}
$(1) \implies (3)$: 
Suppose $E$ splits and consider one such chain of subvarieties. To prove vanishing of cohomology, we apply Bott's theorem.

We must calculate the cohomology $\bigwedge^{i}Q^{*}(t)$.  This means we need to pair the weight $\lambda_{i} + t\lambda_{k+1} \rho$ with all the positive roots, and count the number of pairings that are negative and we must show that there cannot be $i$ such positive roots.  Because only the parameter $t$ is allowed to be negative, the only positive roots which have the possibility to pair negatively with $\lambda_{i} + t\lambda_{k+1} \rho$ are those including $\alpha_{k+1}$.  

For the first examples, suppose that $i < k $.  When $i=k$ the pairings will be slightly different, but the essential argument is the same. We compute the first pairings in non-decreasing order as follows:
\[
\begin{array}{rc
}
\underline{\hspace{2em}\alpha \hspace{2em}} & \underline{\langle \alpha, \lambda_{i}+ t\lambda_{k+1} + \rho \rangle }\\
\alpha_{k+1} &  2t+2\\
  \alpha_{k}+ \alpha_{k+1} & 2t+3\\
  \alpha_{k-1}+ \alpha_{k}+ \alpha_{k+1} & 2t+4 \\
  2\alpha_{k}+ \alpha_{k+1} & 2t+4 \\
 \alpha_{k-2}+ \alpha_{k-1} +  \alpha_{k}+ \alpha_{k+1} &2t+5\\
 \alpha_{k-1} + 2 \alpha_{k}+ \alpha_{k+1} &2t+5\\
  \alpha_{k-3} +  \alpha_{k-2} +  \alpha_{k-1} +  \alpha_{k}+ \alpha_{k+1} &2t+6\\
\alpha_{k-2} +  \alpha_{k-1} +  2\alpha_{k}+ \alpha_{k+1} &2t+6\\
2\alpha_{k-1} +  2\alpha_{k}+ \alpha_{k+1} &2t+6\\
\vdots & \vdots \\
\end{array}
\]
So, to have $H^{1}(\bigwedge^{1}Q^{*}(t)) \neq 0$ we would need to have precisely one negative pairing and no zero pairings.  The pairing with $\alpha_{k+1}$ yields $ 2t+2 <0$  implying that $t<-1$.  But if $t\leq -2$ then the pairing with $\alpha_{k}+ \alpha_{k+1}$ yields $2t+3 \leq -1$, so $H^{1}(\bigwedge^{1}Q^{*}(t)) = 0$, and the nonzero cohomology can only occur in degree at least $2$.

Similarly, to have $H^{2}(\bigwedge^{2}Q^{*}(t)) \neq 0$ we need precisely two negative pairings, and no zero pairings. So we must have $2t+3 <0$ or $t\leq -2$.  If $t=-2$ then we get a zero pairing.  If $t <-2$ then we would have more than two negative pairings.

In general, one checks that for each integer $i \leq k+1$, there is at least one positive root $\alpha$ such that $\langle\alpha, \lambda_{i}+t \lambda_{k+1}+\rho \rangle = 2t + j$ for all $2 \leq j \leq 2k+2$.  So to require at least $i$ negative pairings would also imply either that the cohomology is singular, or that there are strictly more than $i$ negative pairings.  Therefore the cohomology of $\bigwedge^{i}Q^{*}(t)$ is either singular, or occurs in degree greater than $i$.  Since $Q$ has rank $k+1$ we conclude more than we needed to show, namely that $H^{d}(\bigwedge^{i}Q^{*}(t)) = 0$ in the range $1\leq d \leq 2k+2$.

$(3) \implies (2)$: The existence of such a chain of subvarieties is constructed as in the proof of Theorem \ref{sufficient}: Consider a generic section $s$ of the quotient bundle $Q^{(k)}$ on $LG(k)$.  Then we showed that $zeros(s) = LG(k-1)$. Iterate.

$(2) \implies (1)$:
Consider the restriction $E_{|L(k-1)}$.  We do not know if this vector bundle splits or not.  If $E_{|L(k-1)}$ splits it is isomorphic to a bundle $\bigoplus_{j}\OO_{L(k-1)}(a_{j})$ for some constants $a_{j}$. Let $F = \bigoplus_{j} \OO_{LG(k)}(a_{j})$ so that $F_{|L(k-1)} \cong
E_{|L(k-1)}$.  By the same argument as in the proof of Theorem \ref{sufficient}, we know that the isomorphism between $E$ and $F$ on $L(k-1)$ lifts if $H^{1}(\II_{L(k-1)} \otimes F^{*} \otimes E) =0$.  By applying the Koszul resolution and Lemma 1.1 of \cite{Ott89}, this vanishing can be guaranteed if $H^{j}(L(k-1),\bigwedge^{j}Q^{(k)}\otimes E _{LG(k-1)}(t)) =0$ for all $t\in \zz$ and for all $1\leq j\leq k+1$.

In a similar manner, we can consider the restriction $(E_{|L(k-1)})_{|L(k-2)}$.  If this bundle splits, then we will need to ask for the vanishing of $H^{j}(L(k-2),\bigwedge^{j}Q^{(k-1)}\otimes E_{LG(k-2)} (t)) =0$, $t\in \zz$ and for all $j\leq k$ in order to guarantee that the isomorphism between the splitting bundle and $E$ at the level of $LG(k-2)$ lifts.  
We continue to descend until we get to $LG(0)= \pp^{1}$.  At this base level, we know that every vector bundle splits over $\pp^{1}$, so we require $H^{j}(\pp^{1},\bigwedge^{j}Q^{(0)}\otimes E _{\pp^{1}}(t)) =0$, $t\in \zz$ and for all $1\leq j\leq 2$.
%In summary, we consider the vector bundle $E$ on $LG(k)$.  Then we consider a chain of subvarieties $LG(k) \supset LG(k-1) \supset \ldots \supset LG(0) = \pp^{1}$, and we consider the restriction of $E$ to each variety in the chain.  We know the bundle restricted to $\pp^{1}$ must split.  So we must lift an isomorphism on $\pp^{1}$ between $E$ and some splitting bundle to an isomorphism on $LG(1)$, and so on. To guarantee lifting of an isomorphism between splitting bundles over $LG(i-1)$ an isomorphism over $LG(i)$, we require vanishing of $H^{j}(L(i-1),\bigwedge^{j}Q^{(i)}\otimes E _{LG(i-1)}(t)) =0$ for all $t\in \zz$ and for all $1\leq j\leq i+1$.
\end{proof}

%%%%%%%%%%%%%%%%%%%%%%%%%%%%%%%%%%%%%%%%%%%%

\textbf{Acknowledgements:}
The authors are grateful for the support of PRAGMATIC and the University of Catania. They wish to thank Giorgio Otavianni and Daniele Faenzi for the fruitful discussions that led to this work.

\bibliographystyle{amsalpha}
%\cleardoublepage
%\addcontentsline{toc}{chapter}{Bibliography}
\bibliography{geoalg}

\end{document}